\documentclass[twoside,reqno,12pt] {amsart}
\usepackage[left=1in,right=1in,top=1in,bottom=1in]{geometry}

\usepackage{amsmath,amssymb,amsthm}
\usepackage{comment,mathtools}
\usepackage[hidelinks]{hyperref}
\usepackage[dvipsnames]{xcolor}

\usepackage{mathrsfs}

\usepackage[top=1in, bottom=1in, left=1in, right=1in]{geometry}

\renewcommand{\Im}{\mathrm{Im}}

\numberwithin{equation}{section}
\newtheorem{theorem}{Theorem}[section]
\newtheorem{proposition}[theorem]{Proposition}
\newtheorem{lemma}[theorem]{Lemma}

\newtheorem{assumption}{Assumption}

\definecolor{darkblue}{rgb}{0,0,0.7}

\newcommand{\R}{{\mathbb R }}
\newcommand{\N}{{\mathbb N}}

\usepackage{soul}

\newcommand{\supp}{\mathop{\mathrm{supp}}}

\newcommand{\donothing}[1]{{}}

\let\OLDthebibliography\thebibliography
\renewcommand\thebibliography[1]{
\OLDthebibliography{#1}
\setlength{\parskip}{1pt}
\setlength{\itemsep}{1pt plus 0.3ex}
}

\newcommand{\M}{\mathcal{M}}

\def \p {\partial}

\title[Partial Data inverse problem for semi-linear wave equation]{On a partial data inverse problem for the semi-linear wave equation}

\author[Liu]{Boya Liu}
\address{Boya Liu, Department of Mathematics\\
North Dakota State University, Fargo\\ 
ND 58102, USA}
\email{boya.liu@ndsu.edu}

\author[Wang]{Weinan Wang}
\address{Weinan Wang, Department of Mathematics\\
University of Oklahoma, Norman\\ 
OK 73019, USA}
\email{ww@ou.edu}

\begin{document}

\begin{abstract}
We show that a partial Dirichlet-to-Neumann map, where the measurement set is arbitrarily small, uniquely determines the time-dependent nonlinearity of order three or higher in a semi-linear wave equation up to natural obstructions on a Lorentzian manifold with boundary. In particular, we do not  impose any geometric or size restrictions on the measurement set. The proof relies on the technique of higher order linearization combined with  the construction of Gaussian beams with reflections on the boundary.
\end{abstract}

\maketitle

\section{Introduction and Statement of Results}

In many physical applications, measurements can only be made on certain portions of the boundary, making partial data inverse problems of great practical significance. In this paper, we show that the time-dependent nonlinearity of order at least three in a semi-linear wave equation can be recovered uniquely in an optimal subset  of a globally hyperbolic Lorentzian manifold with boundary from a partial Dirichlet-to-Neumann map, where measurements are performed on an arbitrarily small subset of the boundary.  The assumption of global hyperbolicity
is crucial to guarantee the well-posedness of the forward problem with small   Dirichlet data. 

The proof of our result relies on two main components, the first of which is  the technique of higher order linearization. Since the seminal work \cite{Kurylev_Lassas_Uhlmann}, nonlinearity has been exploited in the study of inverse problems, and the higher order linearization plays a crucial role in it. The other significant ingredient in our proof is the construction of the Gaussian beam quasimodes. Due to the partial data setting, the integral identity we obtain after carrying out the higher order linearization contains a term over the inaccessible portion of the boundary. In order to show that this term vanishes in a suitable limit, we construct solutions that are small   in a neighborhood of the boundary. This is achieved via the construction of Gaussian beams with reflections at the boundary. 
We first utilize the third order linearization and Gaussian beams to recover the cubic nonlinearity, and  the higher order ones are obtained inductively.

The main contribution of this paper lies in the establishment of a partial data uniqueness result for time-dependent coefficients of a semi-linear wave equation on Lorentzian manifolds, particularly without imposing restrictive assumptions on the measurement set or the geometry. The corresponding problem for linear hyperbolic equations still remains open.

\subsection{Problem setting and the main result}

Let $n\ge 2$, and let $(\mathcal{M},g)$ be a $(1+n)$-dimensional Lorentzian manifold with smooth boundary $\p \M$, where the metric $g$ is of signature $(-,+,+,\cdots, +)$. Throughout this paper, we shall assume  that   $(\M,g)$ is \textit{globally hyperbolic}. It means that $(\M, g)$ is isometric to the product manifold $\R \times M$ with the metric
\[
g=-\beta(t, x)dt^2+g_0(t, x), \quad \text{ for all } (t,x)\in \R \times M.
\]
Here $M$ is a $n$-dimensional   manifold with boundary $\p M$, $\beta$ is a positive smooth  function, and $g_0$ is a Riemannian  metric on $M$ that  depends smoothly on the $t$-variable, see \cite{black1}.  From a physical perspective, the product structure $\R \times M$ with a time-dependent metric $g_0(t,x')$ encompasses a wide class of physically relevant space-times, including cosmological models and stationary black hole exteriors \cite{black}.   
Furthermore, for all $t\in \R$, the set $\{t\}\times N$ is a Cauchy hypersurface in $\M$, i.e., any causal curve intersects it at most once \cite{Feizmohammadi_Oksanen_semilinear_Euclidean}. 

When $(0,T)\times M$ is a globally hyperbolic manifold,  its lateral boundary $\Sigma := (0,T)\times \p  M$ is   time-like. In this paper we assume that $\Sigma$ is null-convex, meaning that the second fundamental form II$(V,V)=g(\nabla_V\nu, V)\ge 0$ for all null vectors $V\in T(\Sigma)$. We refer readers to \cite{Hintz_Uhlmann} for a discussion of this condition.

Consider a time-dependent function $V: (0,T)\times M \times \mathbb{C} \rightarrow \mathbb{C}$ that satisfies the following assumptions:
\begin{assumption}
\label{assump1}
The map $\mathbb{C} \ni z \mapsto V(\cdot, z)$ is holomorphic with values in the H\"older space $C^{\alpha}((0,T)\times M)$ for some $0<\alpha<1$.
\end{assumption}

\begin{assumption}
\label{assump2}
$V(t,x, 0)=\partial_{z} V(t,x, 0)=0$ for all $(t,x) \in (0,T)\times M$.
\end{assumption}
\noindent Then it follows from Assumptions \ref{assump1} and \ref{assump2} that $V$ can be expressed as a power series
\begin{equation}
\label{eq:expansion_V}
V(t,x, z)=\sum_{k=2}^{\infty} V_{k}(t,x) \frac{z^{k}}{k!}, \quad V_{k}(t,x):=\partial_{z}^{k} V(t,x, 0) \in C^\infty((0,T)\times M),
\end{equation}
which converges in the $C^{\alpha}$--topology. Additionally,  we assume that $V_2(t,x)=0$ in $(0,T)\times M$ throughout this paper.

Consider the initial boundary value problem for semi-linear wave equation
\begin{equation}
\label{eq:ibvp_semilinear_wave}
\begin{cases}
\Box_g u  + V(t,x,u)=0 & \text{ in }  (0,T)\times M,
\\
u=f & \text{ on } \Sigma,
\\
u(0,\cdot)=0, \quad \p_t u(0,\cdot)=0 & \text{ in } M.
\end{cases}
\end{equation}
Here $\Box_g =\p_t^2 - \Delta_g$ is the wave operator on $(0,T)\times M$. In local coordinates, it acts on $C^2$-smooth functions according to the   expression
\[
\Box_g=|g|^{-1/2}\p_{j}\left(g^{jk}|g|^{1/2}\p_{k}v\right).
\]
Here $|g|$ and $g^{jk}$ denote the absolute value of the determinant and the inverse of $g_{jk}$, respectively. Global hyperbolicity of the Lorentzian manifold provides a crucial geometric framework to establish the well-posedness of \eqref{eq:ibvp_semilinear_wave} with a  small Dirichlet value $f$. When a small Neumann boundary condition is considered, the work \cite{Hintz_Uhlmann_Zhai} established the well-posedness of \eqref{eq:ibvp_semilinear_wave}, see also \cite{Lassas_Liimatainen_Potenciano_Tyni} and the references therein for additional discussions  on the well-posedness of nonlinear hyperbolic equations. 

We next introduce the boundary measurement considered in this paper. Let $\Gamma$ be an arbitrary nonempty open proper subset of $\Sigma$. Associated with the problem \eqref{eq:ibvp_semilinear_wave}, we define the partial Dirichlet-to-Neumann map $\Lambda^\Gamma_V$ by the formula
\begin{equation}
\label{eq:def_DN_map}
\Lambda^\Gamma_V f = \p_\nu u|_{\Gamma}. 
\quad 
\supp f \subseteq \Gamma.
\end{equation}
Here $f\in C^{m+1}(\Sigma)$ is small for some  large fixed number $m$, the function $u$ is the unique solution to the problem \eqref{eq:ibvp_semilinear_wave}, and $\nu$ is the unit outer normal to the boundary. 

Let us now recall some notations and definitions related to Lorentzian manifolds by following \cite{Beem_Ehrlich_Easley,ONeill}. Let $\M$ be a Lorentzian manifold. A smooth curve $\mu: (a,b)\to \M$ is said to be \textit{time-like} if $g(\dot{\mu}(s), \dot{\mu}(s)) < 0$ for all $s\in (a,b)$, and is called \textit{causal} if $g(\dot{\mu}(s), \dot{\mu}(s))\le 0$ and $\dot{\mu}(s)\ne 0$ for all $s\in (a,b)$. For any points $p,q\in \M$, we write $p\ll q$ if $p\ne q$ and there exists a future-pointing time-like path from $p$ to $q$. Similarly, we denote $p<q$ if $p\ne q$, and they can be joined by a future-pointing causal path. Furthermore, we write $p\le q$ if $p=q$ or $p<q$. Then the \textit{chronological future} of $p\in \M$ is the set $I^+(p)=\{q\in \M: p\ll q\}$, and the \textit{causal future} of $p$ is given by the set $J^+(p)=\{q\in \M: p\le q\}$. The \textit{chronological past} $I^-(q)$ and \textit{causal past} $J^{-}(q)$ of $q\in \M$ are defined in similar manners. 

For any set $A\subset \M$, we define $J^\pm (A) = \cup_{p\in A} J^\pm (p)$, which is always open. On the other hand, if $(\M, g)$ is globally hyperbolic, by \cite[Lemmas 14.6 and 14.22]{ONeill}, the sets $J^\pm (p)$ are closed, and the sets $I^\pm (p)$ and $J^\pm (p)$ satisfy the relation $\mathrm{cl}(I^\pm(p)) = J^{\pm}(p)$. Here $\mathrm{cl}(A)$ means the closure of the set $A$. 

Due to the finite speed of wave propagation and the causal structure of the Lorentzian manifold, it is only possible to recover $V(t,x)$ in the set 
\[
\mathbb{U} =  \bigcup_{p,q \in \Sigma} I^-(q) \cap I^+(p).
\]
This is the domain that can be reached by sending the waves from $\Sigma$ so that the possible signals generated by a nonlinear interaction of waves can also be detected on $\Sigma$. We refer readers to see \cite[Figure 1]{Lassas_Uhlmann_Wang} for a visualization of   $\mathbb{U}$. 
In this paper we assume that  null-geodesics do not have conjugate points in $\mathbb{U}$ or on $\Sigma$.

Our main result of this paper is as follows.

\begin{theorem}
\label{thm:main_result}
Let $(\M,g)$ be a  $(1+n)$-dimensional globally hyperbolic Lorentzian manifold with $n\ge 2$, and let $\Sigma = (0,T)\times \p M$ be its lateral boundary. Let $\Gamma \subset \Sigma$ be an arbitrary non-empty open set. Assume that null-geodesics do not have conjugate points in $\mathbb{U}$ or on $\Sigma$. Let $V^{(1)}, V^{(2)}: (0,T)\times M \times \mathbb{C} \rightarrow \mathbb{C}$ satisfy  Assumptions \ref{assump1} and \ref{assump2}. Furthermore, assume that $V_2^{(j)}(t,x)=0$, $j=1,2$, for all $(t,x)\in (0,T)\times M$ in the expansion \eqref{eq:expansion_V}.
Then $\Lambda_{V^{(1)}}^{\Gamma} =\Lambda_{V^{(2)}}^{\Gamma} $ implies that $V^{(1)}=V^{(2)}$ in $\mathbb{U}$.
\end{theorem} 

As noted in \cite{Kurylev_Lassas_Uhlmann},  nonlinearity makes it possible to solve inverse problems for certain nonlinear equations whose linear counterparts remain open. 
The  result corresponding to Theorem \ref{thm:main_result} for the linear wave equation with time-dependent coefficients is not available. To our understanding, the best partial data results in this direction are established in \cite{Kian_partial_data} in the Euclidean setting and in \cite{Kian_Oksanen,Liu_Saksala_Yan_potential} on Riemannian manifolds having a certain product structure. In these results, the Dirichlet and Neumann data  are both measured on approximately half of the lateral boundary of the space-time. 

Finally, let us remark that  the unique  recovery of the Lorentzian metric and the quadratic nonlinearity (up to a natural gauge) from the partial Dirichlet-to-Neumann map \eqref{eq:def_DN_map} requires a different approach from the one developed in this paper. We shall leave these questions to future works, see \cite{Hintz_Uhlmann_Zhai_4th_order,Hintz_Uhlmann_Zhai} and the references therein for several full data results in this direction, as well as  a discussion of different methods to address the full data problem. 

\subsection{Previous literature}
Inverse problems for hyperbolic equations have attracted significant attention in recent years. In this paper, we restrict our discussions only to the known results related to nonlinear hyperbolic equations. For literature concerning  linear hyperbolic equations, we refer readers to   \cite{Liu_Saksala_Yan,Liu_Saksala_Yan_potential} for a review of uniqueness results, as well as  \cite{liu2025h} for a survey of works addressing  the issue of  stability.

The pioneering work   \cite{Kurylev_Lassas_Uhlmann} established fundamental results for semi-linear wave equations on Lorentzian manifolds. It also shows that nonlinearity is a beneficial tool to study inverse problems for nonlinear  equations.  There are many subsequent works devoting to the study of  inverse problems for nonlinear hyperbolic equations. Inverse boundary value problems of recovering the metric or the nonlinearity are studied for semi-linear or quasi-linear wave equations in \cite{Acosta_Uhlmann_Zhai,Hoop_Uhlmann_Wang,Hintz_Uhlmann,Hintz_Uhlmann_Zhai_4th_order,Hintz_Uhlmann_Zhai,Li_Zhang,Uhlmann_Zhai,Uhlmann_Zhang_Quadratic,Uhlmann_Zhang_acoustics,Zhang_damping} and the references therein. We also refer readers to \cite{Uhlmann_Zhai_survey} for a survey of some recent progress on inverse problems for nonlinear hyperbolic equations.

Another important type of data appearing in the study of inverse problems for hyperbolic equations is the \textit{source-to-solution} map, which is commonly used in the setting of manifolds without boundary. For linear equations, this problem has been investigated on closed or complete Riemannian manifolds in \cite{source-to-solution,Lassas_et_al_disjoint,Lassas_Nul_Oksanen_Ylinen,saksala2025inverse}, among others. In the nonlinear setting, the works \cite{Kurylev_Lassas_Uhlmann,Lassas_Uhlmann_Wang} recovered a Lorentzian metric and the nonlinearity  (up to their natural gauge)  from the source-to-solution map associated with a semi-linear wave equation. Subsequently, the authors of \cite{Feizmohammadi_Oksanen_semilinear_Euclidean} established the unique recovery of a time-dependent linear zeroth order perturbation appearing in a semi-linear wave equation with cubic nonlinearity. Moreover, several  uniqueness results showing that the source-to-solution map associated with various semi-linear and quasi-linear wave equations  determines the metric up to a natural gauge were established in \cite{Feizmohammadi_Lassas_Oksanen_stos}, notably with  sources and the receivers supported on disjoint subsets of a Lorentzian manifold. 

\subsection{Main ideas of the proof}


Compared with the linear case, inverse problems for  nonlinear equations present different challenges.
Although   nonlinearity is beneficial to the study of inverse problems,    it   introduces significant technical difficulties in establishing the well-posedness of the forward problem. We shall establish the well-posedness for the forward problem \eqref{eq:ibvp_semilinear_wave} in Section \ref{sec:wellposedness} with a sufficiently smooth small Dirichlet value. 
The main idea in the proof is to transform the   problem \eqref{eq:ibvp_semilinear_wave} into a similar problem with zero Dirichlet boundary value. Then we show the existence and uniqueness of solutions via a contraction mapping argument. The global hyperbolicity assumption of the Lorentzian manifold allows us to foliate the space-time into Cauchy hypersurfaces and provides a natural time function \cite{black1}, which is essential for our analysis of the forward problem.

There are two main ingredients in the proof of Theorem \ref{thm:main_result}, the first being the technique of higher order linearization. A main reason that nonlinearity is beneficial to solve inverse problems for nonlinear wave equations is that the interaction of several nonlinear waves   produces new waves, which give information that is not possible to obtain for linear equations. The technique of higher order linearization plays a significant role in this treatment. We refer readers to \cite[Section 1]{Uhlmann_Zhai_survey} for a brief overview of this technique. Since its introduction in  \cite{Kurylev_Lassas_Uhlmann}, this technique has been widely applied in the study of inverse problems for various types of nonlinear equations. Among the vast amount of related literature, we refer readers to  \cite{Carstea_Nakamura_Vashisth,Feizmohammadi_Oksanen_elliptic,Krupchyk_Uhlmann_gradient,Krupchyk_Uhlmann_semilinear_partial,Krupchyk_Uhlmann_CTA,Lassas_Liimatainen_Lin_Salo} for applications to nonlinear elliptic equations,   \cite{Feizmohammadi_et_all_2019,Hintz_Uhlmann_Zhai,Lassas_Liimatainen_Potenciano_Tyni,Lassas_Uhlmann_Wang,Wang_Zhou} to nonlinear hyperbolic equations, as well as \cite{Kian_Uhlmann,Lai_Lu_Zhou,Lai_Uhlmann_Yan} to nonlinear parabolic equations.

Assumption \ref{assump1} of the nonlinear term $V(t,x,u)$ is essential for the higher-order linearization approach, as it allows us to differentiate the nonlinear terms arbitrarily many times. On the other hand, Assumption \ref{assump2} ensures that homogeneous Dirichlet boundary data leads to zero solutions, which simplifies the linearization procedure.  

The construction of Gaussian beam solutions to the linear wave equation, which is accomplished in Section \ref{sec:Gaussian_beam}, is another major component in our proof. In particular, we use a reflection argument to construct the Gaussian beams, which are small near the boundary. The reason of such construction is as follows. Due to the partial data setting, the integral identity \eqref{eq:integral_identity}, which we obtain after performing the higher order linearization, contains a term over the inaccessible portion of the boundary. The key difficulty of the inverse problem considered in this paper lies in this part. In order to overcome it, we shall construct Gaussian beams satisfying desired boundary behaviors, so that the boundary term vanishes in a certain limit.  The construction of Gaussian beams with reflection at the boundary was first developed in \cite{Kenig_Salo} in the setting of Riemannian manifolds and  \cite{Hintz_Uhlmann_Zhai_4th_order} for Lorentzian manifolds. This paper  further develops these techniques with the geometric framework of global hyperbolicity. 

To complete the proof of Theorem \ref{thm:main_result}, by  utilizing the third order linearization and the Gaussian beam solutions, we first show in Subsection \ref{subsec:recovery_V3}  that the cubic nonlinearity $V_3$ in \eqref{eq:expansion_V} can be uniquely recovered  in $\mathbb{U}$ from the partial Dirichlet-to-Neumann map $\Lambda_V^\Gamma$. Then in Subsection \ref{subsec:recovery_higher_order} we establish the uniqueness of $V_m$, $m\ge 4$, by means of higher order linearization and an induction argument. 


This paper is organized as follows. Section \ref{sec:wellposedness} establishes the well-posedness of the initial boundary value problem \eqref{eq:ibvp_semilinear_wave} with small Dirichlet boundary data. Then  we construct Gaussian beam solutions to the linear wave equation in Section \ref{sec:Gaussian_beam}. Finally, we utilize these  solutions and perform higher order linearization to prove Theorem \ref{thm:main_result} in Section \ref{sec:proof_main_result}. 

\section{Well-posedness of the Forward Problem}
\label{sec:wellposedness}

In this section, we prove the well-posedness of the forward problem \eqref{eq:ibvp_semilinear_wave} with small Dirichlet boundary data. For the case of small Neumann boundary data, the corresponding well-posedness was established in \cite{Hintz_Uhlmann_Zhai}. 

Our main result in this section is as follows.

\begin{theorem}
\label{thm:wellposedness}
Let $s \ge 5$, and let $\varepsilon_0>0$ be a small number. For any function $f\in C^{s+1}(\Sigma)$ such that $\|f\|_{C^{s+1}(\Sigma)} \leq \varepsilon_0$ and satisfies necessary compatibility conditions $\frac{\partial^{\ell} f}{\partial t^{\ell}}(0,\cdot)=0$ for $\ell=1,2,3,...,s$, the initial boundary value problem \eqref{eq:ibvp_semilinear_wave} admits a unique solution   $u \in \bigcap_{k=0}^s C^k\left([0, T] ; H^{s-k}(M)\right)$. 
\end{theorem}

\begin{proof}

We begin by extending the boundary data $f$ to the interior of the manifold. By standard extension theory, see for instance \cite[Lemma 5.34]{lee2012smooth}, there exists a function $h \in C^{s+1}([0, T] \times M)$ such that $h|_{\Sigma}=f$ and satisfies the inequality
\begin{equation}
\label{eq:est_solution_boundary_condition}
\|h\|_{C^{s+1}((0, T) \times M)} \le C\|f\|_{C^{s+1}(\Sigma)}
\le
C\varepsilon_0.
\end{equation}

Let $\widetilde{u}=u-h$, where $u$ is a solution of the   problem \eqref{eq:ibvp_semilinear_wave}. Then it is straightforward to see that $\widetilde{u}$ is a solution to the  problem
\[
\begin{cases}
\square_g \widetilde{u}= -\square_g h - V(t,x, \widetilde{u}+h) & \text { in }(0,T) \times M,
\\
\widetilde{u}=0 & \text { on } \Sigma,
\\
\widetilde{u}(0,\cdot)=\partial_t \widetilde{u}(0,\cdot)=0 & \text { in } M.
\end{cases}
\]
Let us write $-\square_g h - V(t,x, \widetilde{u}+h)=\mathscr{F}(t,x, h)+G(t,x, \widetilde{u}, h) \widetilde{u}$, where $\mathscr{F}(t,x,h)=-\square_g h-V(t,x, h)$ and
\begin{equation}
\label{eq:def_G}
G(t,x, \widetilde{u}, h)=-\int_0^1 \partial_z V(t,x, h+\tau \widetilde{u}) d \tau.
\end{equation}

According to the expansion \eqref{eq:expansion_V}, we see that $V(t,x, z)$ is smooth in the $z$-variable. Hence, it follows from the Sobolev embedding theorem that 
\[
\sup _{t \in(0,T)} \sum_{k=0}^{s-1}\left\|\partial_t^k \mathscr{F}(t, \cdot, h)\right\|_{H^{s-k-1}(\Sigma)} 
\leq 
C \sup _{t \in(0,T)} \sum_{k=0}^{s-1}\left\|\partial_t^k \mathscr{F}(t, \cdot, h)\right\|_{C^{s-k-1}(\Sigma)}. 
\]
Due to the inequality \eqref{eq:est_solution_boundary_condition}, as well as the fact that $\Box_g$ is a second order hyperbolic operator, we obtain the estimate
\[
\|\square_g h\|_{C^{s-1}((0, T) \times M)} \leq C \|h\|_{C^{s+1} ((0, T) \times  M)} \leq C \varepsilon_0.
\]
On the other hand, since $V(t,x, z)$ is smooth and $V(t,x, z) = \mathcal{O}(|z|^2)$ near $z=0$, we have
\[
\|V(\cdot, h)\|_{C^{s-1}((0, T) \times M)} 
\leq 
C \|h\|_{C^{s-1}((0, T) \times M)}^2
\leq 
C\varepsilon_0^2.
\]
and
\begin{equation}
\label{e.w11041}
\partial_z V(t,x,z)=\mathcal O(|z|).
\end{equation}
Therefore, it follows from the previous two inequalities that
\[
\sup _{t \in(0,T)} \sum_{k=0}^{s-1}\left\|\partial_t^k \mathscr{F}(t, \cdot, h)\right\|_{H^{s-k-1}(\Sigma)} 
\le
C\sup _{t \in(0,T)} \sum_{k=0}^{s-1}\left\|\partial_t^k \mathscr{F}(t, \cdot, h)\right\|_{C^{s-k-1}(\Sigma)}
\le
C\varepsilon_0.
\]


For $R>0$, we define the set $Z(R, T)$ as follows:
\[
Z(R, T) = \left\{w \in \bigcap_{k=0}^s W^{k, \infty}\left((0,T) ; H^{s-k}(M)\right):
\|w\|_Z^2:=\sup _{t \in(0,T)} \sum_{k=0}^s \left\|\partial_t^k w(t,\cdot)\right\|_{H^{s-k}(M)}^2 \leq R^2\right\}.
\] 
Since $\partial_z V(t,x, z)$ vanishes linearly in $z$, the function $G$ defined in \eqref{eq:def_G} satisfies
\[
G(t,x, \widetilde{u}, h) \in \bigcap_{k=0}^s W^{k, \infty}\left((0,T) ; H^{s-k}(M)\right).
\]
Furthermore, for any function $\widetilde{u} \in Z\left(\rho_0, T\right)$ with $\rho_0$ sufficiently small, due to the inequalities \eqref{eq:est_solution_boundary_condition} and \eqref{e.w11041}, it holds that  
\[
\|G(t,x, \widetilde{u}, h)\|_Z 
\leq 
\int_0^1 \|\partial_z V(t,x, h+\tau \widetilde{u})\|_Z d \tau
\leq
C\left(\|h\|_Z+\|\widetilde{u}\|_Z\right) \leq C\left(\varepsilon_0+\|\widetilde{u}\|_Z\right).
\]

Given a function $\widetilde{w} \in Z\left(\rho_0, T\right)$,   consider the linear initial boundary value problem
\begin{equation}\label{e.w04201}
\begin{cases}
\square_g \widetilde{u}-G(t,x, \widetilde{w}, h) \widetilde{w}=\mathscr{F}(t,x, h), & t \in(0,T), 
\\
{\color{black}{\widetilde{u}=0}},  & t \in(0,T), 
\\
\widetilde{u}(0,\cdot)=\p_t  \widetilde{u}(0,\cdot)=0 .
\end{cases}
\end{equation}
By \cite[Theorem 3.1]{DH},  this problem has a unique solution $\widetilde{u} \in \bigcap_{k=0}^s C^k \left((0,T) ; H^{s-k}(M)\right)$ that satisfies the estimate
\[
\|\widetilde{u}\|_Z \leq C\left(\varepsilon_0+\varepsilon_0\|\widetilde{w}\|_Z+\|\widetilde{w}\|_Z^2\right) e^{K T},	
\]
where $C$ and $K$ are positive constants depending only on the coefficients of the equation. 

Let us define the map $\mathscr{T}: Z(\rho_0, T)  \to \bigcap_{k=0}^s C^k\left((0,T) ; H^{s-k}(M)\right)$ by  $\mathscr{T}(\widetilde{w}) = \widetilde{u}$, where $\widetilde{u}$ is a solution of the problem \eqref{e.w04201}. By setting $\varepsilon_0=\frac{e^{-K T}}{2 C} \rho_0$ and choosing $\rho_0>0$ small enough such that
\begin{equation}
\label{eq:choice_rho0}
C\left(\varepsilon_0+\varepsilon_0 \rho_0+\rho_0^2\right) e^{K T}<\rho_0,
\end{equation}
we see that $\mathscr{T}$ indeed maps $Z\left(\rho_0, T\right)$ to itself.

We now assume for $j=1,2$ that $\widetilde{u}_j$ satisfies the equation
\[
\begin{cases}
\square_g \widetilde{u}_j-G\left(t,x, \widetilde{w}_j, h\right) \widetilde{w}_j=\mathscr{F}(t,x, h), \quad t \in(0,T),
\\
\widetilde{u}_j(0,\cdot)=\p_t \widetilde{u}_j (0,\cdot)=0.
\end{cases}
\]
Then we get that $\widetilde{u}_j= \mathscr{T} (\widetilde{w}_j)$. Also, it follows from the fundamental theorem of calculus  that
\[
\square_g\left(\widetilde{u}_1-\widetilde{u}_2\right)
=
-\left(\int_0^1 \partial_z V \left(t,x, h+\widetilde{w}_2+\tau\left(\widetilde{w}_1-\widetilde{w}_2\right)\right) d \tau\right)\left(\widetilde{w}_1-\widetilde{w}_2\right).
\]
Thanks to the inequality   \eqref{eq:choice_rho0}, we get
\[
\left\|\mathscr{T} (\widetilde{w}_1)-\mathscr{T} (\widetilde{w}_2)\right\|_Z=\left\|\widetilde{u}_1-\widetilde{u}_2\right\|_Z \leq C\left(\varepsilon_0+\rho_0\right) e^{K T}\left\|\widetilde{w}_1-\widetilde{w}_2\right\|_Z.
\]
By choosing $\rho_0$ sufficiently small such that $C\left(\varepsilon_0+\rho_0\right) e^{K T}<1$, we conclude that the map $\mathscr{T}$ is a contraction. Consequently, the equation \eqref{e.w04201} has a unique solution $\widetilde{u} \in Z\left(\rho_0, T\right)$. Applying \cite[Theorem 3.1]{DH} again, we have
\begin{equation}
\label{eq:def_Em}
\widetilde{u} \in \bigcap_{k=0}^s C^k\left((0,T) ; H^{s-k}(M)\right)=: E^s((0,T)\times M).
\end{equation}
This completes the proof of Theorem \ref{thm:wellposedness}.
\end{proof}

We next state the well-posedness of the linear wave equation, which is given in the following proposition. This result was established in \cite[Proposition 8]{Lassas_Liimatainen_Potenciano_Tyni} on globally hyperbolic Lorentzian manifolds.  

\begin{proposition}
\label{prop:wellposedness_linear_wave}
Let $((0,T) \times M, g)$ be a globally hyperbolic Lorentzian manifold with smooth boundary, and let $\Sigma = (0,T)\times \p M$ be its lateral boundary. Let $s\in \N$. Assume that the functions $F\in E^s((0,T)\times M)$, $f\in H^{s+1}(\Sigma)$, $u_0 \in H^{s+1}(M)$, and $u_1 \in H^s(M)$ satisfy the compatibility condition. Then the equation
\[
\begin{cases}
\Box_g u=F & \text{ in } (0,T) \times M,
\\
u=f &\text{ on } \Sigma,
\\
u(0,\cdot)=u_0, \quad \p_tu(0,\cdot)=u_1  &\text{ in }  M
\end{cases}
\]
has a unique solution $u\in E^{s+1}((0,T)\times M)$, which satisfies the estimate
\[
\|u\|_{E^{s+1}((0,T)\times M)}
\le
C\left(\|F\|_{E^s((0,T)\times M)}
+
\|f\|_{H^{s+1}(\Sigma)}
+
\|u_0\|_{H^{s+1}(M)}
+
\|u_1\|_{H^{s}(M)}\right).
\]
Furthermore, we have $\p_\nu u|_{\Sigma} \in H^s(\Sigma)$.
\end{proposition}

\section{Construction of Gaussian Beam Solutions}
\label{sec:Gaussian_beam}

Our goal of this section is to  construct  solutions of the linear wave equation 
\begin{equation}
\label{eq:linear_wave_forward}
\begin{cases}
\Box w  = 0 & \text{ in }  (0,T)\times M ,
\\
w=f & \text{ on }  \Sigma,
\\
w(0,\cdot)=0, \quad \p_t w(0,\cdot)=0 & \text{ in }   M,
\end{cases}
\end{equation}
which are of the form
\[
w(t,x)=e^{i\rho \varphi(t,x)}a_{\rho} (t,x)+r_{\rho}(t,x).
\]
We shall also construct a solution $w_0$ satisfying the backward wave equation
\begin{equation}
\label{eq:linear_wave_backward}
\begin{cases}
\Box w_0 = 0  & \text{ in }  (0,T)\times M,
\\
w_0=f_0 &\text{ on }  \Sigma,
\\
w_0(T,\cdot)=0, \quad \p_t w_0(T,\cdot)=0 & \text{ in }  M,
\end{cases}
\end{equation}
given by 
\[
w_0(t,x)=e^{i\rho \varphi_0(t,x)}a_{\rho,0} (t,x)+r_{\rho,0}(t,x).
\]
Here the principal term $v_\rho := e^{i\rho \varphi(t,x)}a_\rho (t,x)$, called the Gaussian beam,  concentrates on a null geodesic $\gamma: (a-\delta, b+\delta) \to (0,T)\times M$, $r_{\rho}$ is a remainder term that vanishes in a suitable sense as $\rho \to \infty$. In the Gaussian beam $v_\rho$, $\varphi$ is a complex-valued phase function, and $a_{\rho}$ is a smooth amplitude term.

We shall utilize a reflection argument to construct the Gaussian beams. This approach was  developed in \cite{Kenig_Salo} in the setting of Riemannian manifolds, and \cite{Hintz_Uhlmann_Zhai_4th_order} for globally hyperbolic Lorentzian manifolds. We take this approach for the following reason:
Since the Neumann data is only measured on parts of the boundary, the integral identity \eqref{eq:integral_identity} contains an integral over the inaccessible portion of the boundary. Hence, we need to construct a solution $w_0$ so that this integral vanishes at the limit $\rho \to \infty$. In partial data inverse problems for the linear wave equation,   the authors of \cite{Liu_Saksala_Yan_potential} utilized Carleman estimates to show that such terms vanish in a suitable limit. However, this approach relies on the Riemannian manifold having a specific product structure, see \cite{Ferreira_Kenig_Salo_Uhlmann}. It also requires that the Neumann data is measured on approximately half of the lateral boundary of the space-time. In contrast, our method requires only minimal geometric assumptions on the Lorentzian manifold or the measurement set.

The construction will be presented in the following two subsections. In Subsection \ref{subsec:Gaussian_beam_w123} we first consider the case where the Gaussian beams do not reflect on $\Sigma$ and present the construction of the phase function, amplitude, and remainder.  Then in Subsection \ref{subsec:construction_w0} we incorporate boundary reflections and construct the Gaussian beam quasimodes that are small near $\Sigma$. 

\subsection{Construction of the phase function, amplitude, and remainder}
\label{subsec:Gaussian_beam_w123}

We first discuss the construction of the phase function $\varphi(t,x)$, the amplitude $a_\rho(t,x)$, as well as the remainder term $r(t,x)$. The construction mainly follows from the ideas presented in \cite[Section 3]{Feizmohammadi_et_all_2019} and \cite[Section 2.4]{KKL_book}, and we provide the detailed construction for the sake of completeness. The proof is rather long and technical. Thus,  we  split it into the following steps.

\subsubsection{Fermi coordinates}
\label{subsec:Fermi_coordinates}

Let us begin the construction by recalling Fermi coordinates near a null geodesic $\gamma$.

\begin{lemma}
\label{lem:Fermi_coordinates}
\cite[Lemma 1]{Feizmohammadi_Oksanen_semilinear_Euclidean} Let $(\M,g)$ be a $(1+n)$-dimensional Lorentzian manifold with $n \ge 2$. Let $\delta>0$, $a<b$, and let $\gamma: (a-\delta, b+\delta) \to \mathcal{M}$ be a null geodesic on $\M$. Then there exists a coordinate neighborhood $(U,\Phi)$ of $\gamma([a,b])$, with the coordinates denoted by $(z_0=s, z_1, \dots, z_n)$, such that
\begin{enumerate}
\item[(i)] $\Phi(U) = (a-\delta', b+\delta') \times B(0,\delta')$, where $B(0,\delta')$ is a ball in $\R^n$ with a sufficiently small radius $\delta'>0$.

\item[(ii)] $\Phi(\gamma(s)) = (s,0,\dots, 0)$.
\end{enumerate}
Moreover, in this coordinate system, the metric $g$ can be written as
\[
g|_\gamma = 2ds dz_1 + \sum_{\alpha=2}^n (dz_\alpha)^2,
\]
and $\p_i g_{jk}|_{\gamma} = 0$ for $i,j,k=0,\dots, n$.
\end{lemma}

The Fermi coordinates are constructed by performing parallel transport of an orthonormal frame along the null geodesic $\gamma$ \cite{gu}. The condition $\p_i g_{jk}|_{\gamma} = 0$ ensures that the Christoffel symbols vanish along $\gamma$, which simplifies subsequent computations. 
Furthermore, this coordinate system is essential for the construction of Gaussian beams, as it allows us to work in a neighborhood where the metric has a simple form along the geodesic.

\subsubsection{Construction of the phase function}
\label{subsec:phase}

Assume that $\gamma$ is a null geodesic. In this subsection we construct the phase function $\varphi$. To this end, we define the set
\[
\Omega 
= 
\left\{ \left(s, z'\right): s \in [a-\delta', b+\delta'], \: |z'|<\delta \right\},
\] 
where $0<\delta <\delta'$ is small enough such that $\Omega$ does not intersect the hypersurfaces $\{0\}\times M$ or $\{T\}\times M$. 

We shall construct a Gaussian beam via the WKB ansatz
\[
v_\rho(s, z';\rho) = e^{i \rho \varphi(s,z')} a_\rho (s,z'),
\]
which approximately solves the wave equation $\Box_g v_\rho=0$ in  Fermi coordinates. Here the functions $\varphi$ and $a_\rho$ are given by
\begin{equation}
\label{eq:form_phase_amplitude}
\varphi=\sum_{j=0}^N \varphi_j\left(s, z^{\prime}\right)
\: \text { and } \: 
a_\rho \left(s, z^{\prime}\right) = \chi\left(\frac{\left|z^{\prime}\right|}{\delta^{\prime}}\right) \sum_{k=0}^N \rho^{-k} b_k\left(s, z^{\prime}\right), 
\end{equation}
where $b_k\left(s, z^{\prime}\right)=\sum_{j=0}^N b_{k, j}\left(s, z^{\prime}\right)$ is written in the Fermi coordinates $z=(s,z')$ near $\gamma$. Here, for all $j, k=0, \ldots, N$, the functions $\varphi_j$ and $v_{k, j}$ are complex-valued homogeneous polynomials of degree $j$ with respect to the variable $z'\in \R^n$, and $\chi(\tau)$ is a non-negative smooth function with compact support such that $\chi(\tau)=1$ for $|\tau| \leqslant \frac{1}{4}$ and $\chi(\tau)=0$ for $|\tau| \geqslant \frac{1}{2}$.

We compute directly that
\[
\square_g \left(e^{i\rho \varphi} a_\rho \right)
= 
e^{i\rho \varphi}
\left[
\rho^2(\mathcal{S} \varphi) a_\rho
- 
i \rho \mathcal{T} a_\rho
+
\square_g a_\rho \right], 
\]
where the differential operators $\mathcal{S}, \mathcal{T}: C^\infty(\M) \to C^\infty(\M)$ are defined by the formulae
\[
\mathcal{S} \varphi=\langle d \varphi,  d \varphi\rangle_g,
\quad 
\mathcal{T} a=2\langle d \varphi,  d a \rangle_g- \left(\square_g \varphi\right) a.
\]
Here we have taken into account the fact that the metric $g$ has signature $(-,+,\ldots, +)$. This computation suggests that we need to construct the phase function $\varphi$ and the amplitude $a_{\rho}$ such that they approximately solve the eikonal equation $\mathcal{S}\varphi=0$ and the transport equation $\mathcal{T}a=0$, respectively. Specifically, we require that   $\varphi$ satisfies the condition that $\mathcal{S}\varphi$ vanishes up to the $N^{\mathrm{th}}$ order along the null geodesic $\gamma$ with respect to the transversal directions. That is, the function $\varphi$ must satisfy the equation
\begin{equation}
\label{eq:eikonal_geodesic}
\frac{\p^\alpha}{\p z^\alpha}(\mathcal{S} \varphi)(s, 0, \dots, 0)=0, \quad  s\in [a-\delta', b+\delta'],
\end{equation}
for all multi-indices $\alpha$ with $|\alpha|\le N$.
We also require that the leading term $b_0$ of the amplitude $a_\rho$ satisfies the following transport equation on $\gamma$: 
\begin{equation}
\label{eq:transport_b0}
\frac{\partial^{\alpha}}{\partial z^{\alpha}}\left(\mathcal{T} b_0\right)(s, 0,\dots, 0)=0, \quad  s\in [a-\delta', b+\delta'].
\end{equation}
Furthermore, the subsequent terms $b_k$, $k=1,\dots,N$, solve the equation
\begin{equation}
\label{eq:transport_bk}
\frac{\partial^{\alpha}}{\partial z^{\alpha}}\left(-i \mathcal{T} b_k+\square_g b_{k-1}\right)(s, 0,\dots, 0)=0, \quad  s\in [a-\delta', b+\delta'].
\end{equation}
In the previous two equations, $\alpha$ is a multi-index such that $|\alpha|\le N$. The equations \eqref{eq:eikonal_geodesic}-- \eqref{eq:transport_bk} ensure that the Gaussian beam $v_\rho$ satisfies the equation $\Box_g v_\rho = \mathcal{O}(\rho^{-N})$ near $\gamma$. 

Let us now proceed to construct the phase function satisfying the conditions
\begin{equation}
\label{eq:properties_phi}
\Im \varphi \ge 0, \quad \Im \varphi|_{\gamma}=0, \quad \Im \varphi (z) \ge C|z'|^2 \: \text{ for all } z\in \Omega.
\end{equation} 
To this end, by the same arguments as in \cite[Subsection 4.2.1]{Feizmohammadi_Oksanen_semilinear_Euclidean},  we take
\[
\varphi_0=0, \quad \varphi_1=z_1, \quad \varphi_2(s, z')=\sum_{1 \leq i, j \leq n} H_{i j}(s) z_i z_j.
\]
Here $H$ is a symmetric matrix with $\Im H(s)>0$ for all $s\in [a-\delta', b+\delta']$. In particular,  $H$ satisfies the Riccati equation
\begin{equation}
\label{eq:Riccati}
\frac{d}{ds} H+H C H+D=0, \quad  
s\in [a-\delta', b+\delta'], \quad 
H(0)=H_0, \text { with } \Im M_0>0,
\end{equation}
where the matrices $C$ and $D$ are  defined by $C_{11}=0$, $C_{i i}=2, i=2,\dots, n$, $C_{i j}=0, i \neq j$, and $D_{i j}=\frac{1}{4}\partial_{i j}^2 g^{11}$. 

The following lemma, which was originally established in \cite[Lemma 2.56]{KKL_book}, provides the existence of solutions to \eqref{eq:Riccati}, see also \cite[Lemmas 3.2 and 3.3]{Feizmohammadi_et_all_2019}. 

\begin{lemma}
\label{lem:solution_Riccati} 
Let $\hat s_0\in [a-\delta', b+\delta']$, and let $H_0$ be a symmetric matrix such that $\Im H(\hat s_0)>0$. Then the Ricatti equation \eqref{eq:Riccati},  with the initial condition $H(\hat s_0)=H_0$, has a unique solution $H(s)$, which is symmetric and $\Im(H(s))>0$ for all $s \in \left(a-\delta', b+\delta'\right)$. Furthermore, we have that $H(s)=Z(s) Y(s)^{-1}$, where the matrices $Y$ and $Z$ solve the first order linear system
\[
\begin{cases}
\frac{d}{ds} Y(s)=C Z(s),  \quad &Y(\hat s_0)=Y_0,
\\
\frac{d}{ds} Z(s)=-D(s) Y(s), \quad &Z(\hat s_0)=H_0 Y_0.
\end{cases}
\] 
In addition, $Y(s)$ is non-degenerate and satisfies the property
\[
\det \left(\Im H(s)\right) |\det(Y(s))|^2
= 
\det \left(\Im H_0(s)\right).
\]
\end{lemma}


Furthermore, we obtain the functions $\varphi_j$, $j\ge 3$, by arguing similarly as in \cite[Subsection 4.2.1]{Feizmohammadi_Oksanen_semilinear_Euclidean}.   

\subsubsection{Construction of the amplitude}
\label{subsec:amplitude}

We next move to seek the amplitude functions $b_k$ given in \eqref{eq:form_phase_amplitude}, and let us start with the leading term $b_0$. To this end, we need to determine the terms $\left\{b_{0, k}\right\}_{k \geqslant 0}$ such that the equation \eqref{eq:transport_b0} holds for all $k= 0, \ldots, N$. 

Let $\gamma$ be a null geodesic. By Lemma \ref{lem:Fermi_coordinates},  we have on $\gamma$ that
\[
-\square_g \varphi=\sum_{i, j=0}^n g^{i j} \partial_{i j}^2 \varphi=\sum_{i=2}^n \partial_{i i}^2 \varphi=\operatorname{Tr}(C H),
\]
where the matrices $C$ and $H$ are the same as in \eqref{eq:Riccati}. 
Thus, when $|\alpha|=0$ in the transport equation \eqref{eq:transport_b0}, using the definition of $\mathcal{T}$  and in view of \eqref{eq:form_phase_amplitude}, it holds on $\gamma$ that
\begin{equation}
\label{eq:transport_b0_expansion}
\begin{aligned}
&2\langle d \varphi,  d b_0 \rangle_g- \left(\square_g \varphi\right) b_0
\\
&= \left[2\p_s b_{0,0}+ \operatorname{Tr} (CH) b_{0,0}\right] 
+
\left[2 \p_s  b_{0,1}+ \operatorname{Tr} (CH)  b_{0,1 }+ \mathcal{E}_1\right]
+
\cdots
+ \mathcal{O}(|z'|^{n+1}).
\end{aligned}
\end{equation}
Here $\mathcal{E}_j$, $j \ge 1$, is a homogeneous polynomial of degree $k$ in the $z'$-variable, and  its coefficients only depend on the functions $\{b_{0,k}\}_{k=0}^{j-1}$ and $\{\varphi_l\}_{l=0}^{j+2}$.

Let us first find the function $b_{0,0}$, which satisfies the equation
\[
2\p_s b_{0,0}+ \operatorname{Tr} (CH)=0.
\] 
It follows from the computations in \cite[Section 3.4]{Feizmohammadi_et_all_2019} that 
\begin{equation}
\label{eq:trace}
\operatorname{Tr}(C H)
=
\p_s \log \left(\operatorname{det} H(s)\right).
\end{equation}
Therefore, we see  that  $b_{0,0}$ satisfies the equation
\[
2 \p_s  b_{0,0}+ \p_s  \log \left(\operatorname{det} H(s)\right)=0, \quad   s \in [a-\delta', b+\delta'].
\]
From here, it is straightforward to deduce that
\[
b_{0,0}(s) 
=
\left(\operatorname{det} H(s)\right)^{-1/2}.
\]

We next turn our attention to find the subsequent functions $b_{0,j}$, $j\ge 1$. Thanks to \eqref{eq:transport_b0_expansion} and \eqref{eq:trace}, it suffices to solve the equation
\[
2 \p_s  b_{0,j}+ \p_s  \log \left(\operatorname{det} H(s)\right) b_{0,j}=-\mathcal{E}_j.
\]
Let us note that this is a first order equation. Thus, it has a unique solution if we impose the zero initial condition at $\hat s_0 \in [a-\delta',b+\delta']$. Therefore, we have obtained the function $b_0$.

Lastly, we construct the amplitudes $a_1, \dots, a_N$ by solving the transport equation \eqref{eq:transport_bk} recursively up to order $N$. On the null geodesic $\gamma$, these equations are similar to those for $b_{0,k}$, but the terms on the right-hand side are nonzero yet still homogeneous in $z'$. Therefore, the derivation follows from analogous arguments to those for $b_{0,k}$. In particular, the functions $b_{j,k}$ are $k^{\mathrm{th}}$ order homogeneous in $z'$ for all $1\le j,k\le N$. We refer readers to \cite[Section 4]{Feizmohammadi_Oksanen_semilinear_Euclidean}  for details.

Furthermore, by \cite[Lemma 2]{Feizmohammadi_Oksanen_semilinear_Euclidean},   the Gaussian beam $v_\rho$ satisfies the following estimate:
\[
\|v_\rho\|_{C((0,T)\times M)} = \mathcal{O}(1).
\]
Also, for any $k\in \N$, it holds that
\begin{equation}
\label{eq:est_Gaussian_conjugated}
\|\Box_g v_\rho\|_{H^k((0,T)\times M)} = \mathcal{O}(\rho^{-K}), \quad K=\frac{N+1}{2}+\frac{n}{4}-k-2.
\end{equation}

%

\subsubsection{Construction of the remainder}
\label{subsec:remainder}

We now move to construct the remainder term $r_{\rho}$. To this end, let $v_\rho$ be a Gaussian beam constructed in the previous subsections. By Proposition \ref{prop:wellposedness_linear_wave}, the initial boundary value problem
\[
\begin{cases}
\Box_g r_\rho = -\Box_g v_\rho & \text{ in } (0,T)\times M,
\\
r_\rho=0 & \text{ on } \Sigma,
\\
r_\rho(0,\cdot)=0, \quad \p_t \rho (0,\cdot)=0 & \text{ in }  M
\end{cases}
\]
admits a unique solution $r_\rho(t,x)\in E^{s+1}((0,T)\times M)$, where the space $E^{s+1}$ was defined in \eqref{eq:def_Em}, such that
\[
\|u\|_{E^{s+1}((0,T)\times M)}
\le
C\|\Box_g v_\rho\|_{E^s((0,T)\times M)}.
\]
An application of \cite[Corollary 11]{Lassas_Liimatainen_Potenciano_Tyni} yields that
\[
\|r_\rho\|_{H^{\ell}((0,T)\times M)}
=
\mathcal{O}(\rho^{-K}),
\]
where $\ell \in \N$ is such that $k>\ell-1+\frac{n+1}{2}$. Here $k$ and $K$ are the same as in the estimate  \eqref{eq:est_Gaussian_conjugated}.
Furthermore, we choose $N$ in \eqref{eq:est_Gaussian_conjugated} sufficiently large and apply the Sobolev embedding theorem to obtain the estimate 
\begin{equation}
\label{eq:est_remainder}
\|r_\rho\|_{C((0,T)\times M)}
=
\mathcal{O}(\rho^{-\frac{n+1}{2}-2}).
\end{equation}


\subsection{Construction of Gaussian beams with reflections at the boundary}
\label{subsec:construction_w0}

In this subsection we incorporate at the boundary and construct Gaussian beams that are small near the boundary.

Let $(\M, g)$ be a $(1+n)$-dimensional Lorentzian manifold with smooth boundary. Suppose that $E \subseteq \p \M$ is a nonempty open set, and let $R=\p \M \setminus E$. We assume that $E$ is the observation set where  geodesics can enter and exit, and $R$ is the reflecting set.

Let $\rho>0$ be a large parameter, and let $\gamma: [\mathbf{t}_0,T] \to \M$ be a broken null-geodesic, where $\mathbf{t}_1<\cdots < \mathbf{t}_N$ are the times of reflections at the boundary such that $\mathbf{t}_j\in (\mathbf{t}_0,T)$ for $j \in =1,\dots, N$. We also assume that $\gamma(\mathbf{t}_0)\in \p \M$ and $\gamma(T)\notin \p \M$. The construction follows from similar arguments as in \cite{Hintz_Uhlmann_Zhai_4th_order}. We shall only provide details of the construction near the first reflection time $\mathbf{t}_1$, and the subsequent reflections can be handled similarly. Consider Gaussian beam quasimode of the form
$$
v_\rho=v_\rho^{\mathrm{inc}}+v_\rho^{\mathrm{ref}},
$$
where $v_\rho^{\text {inc }}$ and $v_\rho^{\text {ref }}$ are the Gaussian beam solutions associated with the geodesic segments $\left.\gamma\right|_{\left[0, \mathbf{t}_1\right]}$ and   $\left.\gamma\right|_{\left[\mathbf{t}_1, \mathbf{t}_2\right]}$, respectively. We shall construct $v_\rho^{\text {inc}}$ and $v_\rho^{\text {ref}}$ such that $\left.\left(v_\rho^{\mathrm{inc}}+v_\rho^{\mathrm{ref}}\right)\right|_{\partial \M}$ is small near $\gamma\left(\mathbf{t}_1\right)$. Following the construction presented in the previous subsection, we have
$$
v_\rho^{\mathrm{inc}}=e^{i \rho \varphi^{\mathrm{inc}}} a_\rho^{\mathrm{inc}} \quad \text{ and } \quad 
v_\rho^{\mathrm{ref}}=e^{i \rho \varphi^{\mathrm{ref}}} a_\rho^{\mathrm{ref}},
$$
where the phase function $\varphi^\bullet$ and the amplitude $a^\bullet$ are given in \eqref{eq:form_phase_amplitude}. Here and for the remainder of this subsection we shall adopt the notation $\bullet$ = inc, ref.

Let $R_1$ be a small neighborhood of $\gamma\left(\mathbf{t}_1\right)$ on $\partial \M $ such that $v_\rho^{\text {inc}}$ and $v_\rho^{\text {ref}}$ are compactly supported in $R_1$. 
Let us choose the phase functions and the amplitudes satisfying the conditions
\[
\varphi^{\text {ref}}|_{\partial \M}=\varphi^{\text {inc}}|_{\partial \M} \quad \text{and} \quad a^{\text {ref}}|_{\partial \M}=- a^{\text {inc}}|_{\partial \M}
\]
up to the $N^{\mathrm{th}}$ order at $\gamma\left(\mathbf{t}_1\right) \in \partial \M$, except that
\[
d \left(\left.\varphi^{\mathrm{inc}}|_{\partial \M}\right)\right|_{\gamma\left(\mathbf{t}_1\right)} =
\left.\dot{\gamma}\left(\mathbf{t}_1-\right)^\flat\right|_{T \partial \M}
\quad   \text{ and } \quad 
d\left(\left.\varphi^{\mathrm{ref}}|_{\partial \M}\right)\right|_{\gamma\left(\mathbf{t}_1\right)} =
\left.\dot{\gamma}\left(\mathbf{t}_1+\right)^\flat\right|_{T \partial \M}.
\]
Here the notation $\flat$ means the musical isomorphism mapping vectors to co-vectors. Thus, we have on $R_1$ that 
\begin{equation}
\label{eq:est_difference_phi_a}
\left|\varphi^{\mathrm{ref}}-\varphi^{\mathrm{inc}}\right| \leq C|y|^{N+1}
\quad \text{ and } \quad 
\left|a_\rho^{\mathrm{ref}} + a_\rho^{\mathrm{inc}}\right| \leq C|y|^{N+1}.
\end{equation}

We next show that for any integer $k\ge 0$, it holds that
\begin{equation}
\label{eq:est_reflection}
\left\|v_\rho\right\|_{H^k\left(R_1\right)} \leq C \rho^{-\frac{N-k+1}{2}-\frac{3}{4}}.
\end{equation}
Let $(\tau, y)$ be the coordinates near $\gamma\left(\mathbf{t}_1\right)$ such that $\partial \M$ is parametrized by the map $y \mapsto(\mathbf{t}_1, y)$ and $\gamma (\mathbf{t}_1)$ corresponds to $(\mathbf{t}_1,0)$. Due to \eqref{eq:properties_phi}, there exists a constant $C'>0$ such that the following inequality holds on $R_1$:  
\begin{equation}
\label{eq:est_phase}
\left|e^{i\rho \varphi^{\bullet}}\right| \leq e^{-C^{\prime} \rho|y|^2}.
\end{equation}


Let us now write $v_\rho$ as
\[
v_\rho
=
\left(e^{i\rho \varphi^{\mathrm{inc}}}-e^{i\rho \varphi^{\mathrm{ref}}}\right) a_\rho^{\mathrm{inc}}+e^{i\rho \varphi^{\mathrm{ref}}}\left(a_\rho^{\mathrm{ref}}+a_\rho^{\mathrm{inc}}\right).
\]
By the Taylor expansion, we have
\[
e^{i\rho \varphi^{\mathrm{inc}}}-e^{i\rho \varphi^{\mathrm{ref}}}=i\rho\left(\varphi^{\mathrm{inc}}-\varphi^{\mathrm{ref}}\right) \int_0^1 e^{i\rho\left(s \varphi^{\mathrm{inc}}+(1-s) \varphi^{\mathrm{ref}}\right)} ds.
\]
Due to the estimates \eqref{eq:est_difference_phi_a} and \eqref{eq:est_phase}, near $y=0$ it holds that
\[
\left|e^{i\rho \varphi^{\mathrm{inc}}}-e^{i\rho \varphi^{\mathrm{ref}}}\right| \leq C \rho|y|^{N+1} e^{-C' \rho|y|^2}.
\]
Thus, we obtain from the previous inequality, as well as the estimates \eqref{eq:est_difference_phi_a} and \eqref{eq:est_phase}, that  
\[
\left|v_\rho|_{R_1}\right|
\leq 
C \rho|y|^{N+1} e^{-C^{\prime} \rho|y|^2}.
\]

We next derive an upper bound for $\p_y^\alpha v_\rho$ on $R_1$, where $\alpha$ is a multi-index such that $|\alpha|\le k$. By a direct computation, we see that
\[
\p_y v_\rho
=
i\rho e^{i\rho \varphi^{\mathrm{inc}}} a_\rho^{\mathrm{inc}}\p_y \varphi^{\mathrm{inc}}
+
e^{i\rho \varphi^{\mathrm{inc}}} \p_y a_\rho^{\mathrm{inc}}
+
i\rho e^{i\rho \varphi^{\mathrm{ref}}} a_\rho^{\mathrm{ref}} \p_y \varphi^{\mathrm{ref}} 
+
e^{i\rho \varphi^{\mathrm{ref}}} \p_y a_\rho^{\mathrm{ref}}.
\]
Due to the construction of the phase function, we get that $\p_y \varphi^\bullet|_{y=0}=0$, which  implies that  the following inequality is valid on $R_1$:
\[
\left|\p_y \varphi^{\bullet}\right|  \leq C|y|.
\]
On the other hand, since $\p_y a_\rho^\bullet$ is smooth in $y$ near $y=0$, and $\p_y^\alpha a_\rho^\bullet|_{y=0}=0$  for all $|\alpha|\le N$, we have on $R_1$ that
\begin{equation}
\label{eq:est_amplitude_reflection_derivative}
|\p_y a_\rho^{\bullet}|\le C|y|^N.
\end{equation}
Therefore, we obtain from the inequalities \eqref{eq:est_difference_phi_a},  \eqref{eq:est_phase}, and \eqref{eq:est_amplitude_reflection_derivative} that
\[
\left|\p_y v_\rho|_{R_1}\right|
\leq 
C \left( \rho |y|^{N+2} e^{-C' \rho|y|^2}
+
|y|^N e^{-C' \rho|y|^2}\right).
\]
Furthermore, by induction, it holds that 
\[
\left|\partial_y^\alpha v_\rho|_{R_1} \right|
\leq 
C \sum_{k+j=|\alpha|} \rho^k y^k e^{-C^{\prime} \rho|y|^2}|y|^{N+1-j}.
\]
Finally, we perform a change of variables $y\mapsto \rho^{-1 / 2} y$ to get that
\[
\int_{R_1}\left|\partial_y^\alpha v_\rho\right|^2 dS_g dt 
\leq 
C \rho^{-(N-|\alpha|+1)-\frac{3}{2}},
\]
from which the estimate \eqref{eq:est_reflection} follows. 

Since all reflection points are distinct, we conclude that the Gaussian beam $v_\rho$ satisfies the following estimates:
\begin{equation}
\label{eq:est_quasimode_inaccessible}
\left\|v_\rho\right\|_{H^k\left(R\right)} =\mathcal{O} \left(\rho^{-\frac{N-k+1}{2}-\frac{3}{4}}\right),
\end{equation}
and
\[
\|v_\rho\|_{L^2((0,T)\times M)}=\mathcal{O}(1), \quad 
\left\|v_\rho\right\|_{L^2((0,T)\times M)}= \mathcal{O}(\rho^{-K}), \quad K=\frac{N+1}{2}+\frac{n}{4}-k-2.
\]
Moreover, we argue similarly as in Subsection \ref{subsec:remainder} to obtain a remainder term $r_\rho$ satisfying the estimate \eqref{eq:est_remainder}.

\section{Proof of Theorem \ref{thm:main_result}}
\label{sec:proof_main_result}

In this section we   prove Theorem \ref{thm:main_result}.  The proof consists of several steps. We first perform a third order linearization and derive an integral identity in Subsection \ref{subsec:linearization}, followed by utilizing the Gaussian beam solutions constructed in Section \ref{sec:Gaussian_beam} to uniquely recover the cubic nonlinearity $V_3$ in Subsection \ref{subsec:recovery_V3}.  Finally, we show, via an inductive argument,  that $V_j$, $j\ge 4$, are uniquely determined by the partial  Dirichlet-to-Neumann map $\Lambda_V\Gamma$.

\subsection{Third order linearization}
\label{subsec:linearization}

In this subsection we perform a third order linearization.  
Let $\varepsilon=\left(\varepsilon_{1}, \varepsilon_{2}, \varepsilon_{3}\right) \in \mathbb{C}^{3}$, and consider the initial boundary value problem \eqref{eq:ibvp_semilinear_wave} with the Dirichlet value $f= \varepsilon_{1} f_{1}+\varepsilon_{2} f_{2}+\varepsilon_{3} f_{3}$. Let $u_\varepsilon=u_\varepsilon(t,x; \varepsilon) \in  E^s((0,T)\times M)$,  $m\geq 5$, be the unique  solution of the initial boundary value problem
\begin{equation}
\label{eq:eq_3rd_linearization}
\begin{cases}
\Box_g u_\varepsilon 
+\sum_{k=3}^{\infty} V_{k}(t, x) \frac{u_\varepsilon^{k}}{k!}=0  & \text { in } (0,T) \times M,
\\
u_\varepsilon = \varepsilon_{1} f_{1}+\varepsilon_{2} f_{2}+\varepsilon_{3}f_3 & \text { on }  \Sigma,
\\
u_\varepsilon(0,\cdot)=0, \quad \p_t u_\varepsilon(0,\cdot)=0 & \text{ in }  M.
\end{cases}
\end{equation}
Let us note that when $\varepsilon=0$, the function $u_\varepsilon(t,x;0)$ satisfies the linear wave equation $\Box_g u_\varepsilon(t,x;0)=0$ with homogeneous initial and boundary conditions. 
Then it follows immediately that $u_\varepsilon(t,x;0) =0$.
Therefore, by differentiating   \eqref{eq:eq_3rd_linearization} with respect to $\varepsilon_{i}$, $i=1,2,3$,  we get that $\p_{\varepsilon_{i}} u_{\varepsilon}|_{\varepsilon=0}=w_i$, where $w_i$, $i=1,2,3$, are solutions of the  initial boundary value problem  for the linear wave equation \eqref{eq:linear_wave_forward} with the  Dirichlet condition $f=f_i$. 

Let us next perform the second order linearization of  \eqref{eq:eq_3rd_linearization}. To this end, we observe that, for any $i,j\in \{1,2\}$,   each term in the sum $ \p_{\varepsilon_{i}} \p_{\varepsilon_{j}}\left(\sum_{k=3}^{\infty} V_{k} (t, x) \frac{u^{k}_\varepsilon}{k!}\right)\big|_{\varepsilon=0}$ contains a positive power of $u_\varepsilon$, which vanishes when $\varepsilon=0$.
Hence,   the function $\mathcal{U}^{(ij)} := \p_{\varepsilon_{i}} \p_{\varepsilon_{j}} u_\varepsilon|_{\varepsilon=0}$ satisfies the initial boundary value problem
\[
\begin{cases}
\Box\mathcal{U}^{(ij)}  = 0  & \text { in }   (0,T) \times M,
\\
\mathcal{U}^{(ij)}=0 & \text { on }   \Sigma,
\\
\mathcal{U}^{(ij)}(0,\cdot)=0, \quad \p_t \mathcal{U}^{(ij)} (0,\cdot)=0  &\text{ in }   M.
\end{cases}
\]

Turning our attention to the third order linearization, we observe that
\[
\p_{\varepsilon_{1}} \p_{\varepsilon_{2}} \p_{\varepsilon_{3}} \left(\sum_{k=4}^{\infty} V_{k} (t, x) \frac{u_\varepsilon^{k}}{k!}\right)\bigg|_{\varepsilon=0}=0,
\] 
as each term in the sum contains a positive power of $u_\varepsilon$, which vanishes when $\varepsilon=0$. For the term $V_3(t,x)\frac{u_\varepsilon^3}{3!}$, it follows from direct computations that
\[
\p_{\varepsilon_{1}} \p_{\varepsilon_{2}}  \p_{\varepsilon_{3}}  \left(V_{3} (t, x) \frac{u_{\varepsilon}^{3}}{3!}\right)\bigg|_{\varepsilon=0} 
= 
V_{3}(t, x)   w_1w_2w_3.
\] 
Thus,  we conclude that the function $\mathcal{U}^{(123)} := \p_{\varepsilon_{1}} \p_{\varepsilon_{2}} \p_{\varepsilon_{3}} u_\varepsilon|_{\varepsilon=0}$ solves the following  problem:
\begin{equation}
\label{eq:eq_3rd_differentiation}
\begin{cases}
\Box\mathcal{U}^{(123)} +V_{3}(t, x)   w_1w_2w_3 = 0 & \text { in }   (0,T)\times M,
\\
\mathcal{U}^{(123)}=0 & \text { on }  \Sigma,
\\
\mathcal{U}^{(123)}(0,\cdot)=0, \quad \p_t \mathcal{U}^{(123)} (0,\cdot)=0 & \text{ in } M.
\end{cases}
\end{equation}
Given  the assumption $\Lambda_{V^{(1)}}^\Gamma (\varepsilon_{1} f_{1}+\varepsilon_{2} f_{2}+\varepsilon_{3} f_{3})= \Lambda_{V^{(2)}}^\Gamma(\varepsilon_{1} f_{1}+\varepsilon_{2} f_{2}+\varepsilon_{3} f_{3})$ for all small $\varepsilon_1, \varepsilon_{2}, \varepsilon_{3}$ and all $f_1,f_2,f_3\in C^{s+1}(\Sigma)$, $s\ge 5$, such that $\supp(f_i)\subset \Gamma$, $i=1,2,3$, we have $\p_\nu u_1 = \p_\nu u_2$ on $\Gamma$, where $u_1$ and $u_2$ are the solutions of the problem \eqref{eq:eq_3rd_linearization} with coefficients $V^{(1)}$ and $V^{(2)}$, respectively. As a consequence, we obtain  $\p_\nu \mathcal{U}^{(1,123)} = \p_\nu \mathcal{U}^{(2,123)}$ on $\Gamma$, where $\p_\nu \mathcal{U}^{(j,123)}=\p_{\varepsilon_{1}} \p_{\varepsilon_{2}} \p_{\varepsilon_{3}} (\p_\nu u_j)|_{\varepsilon=0}$ for $j=1,2$.

Let us next multiply the first equation in \eqref{eq:eq_3rd_differentiation} by the function $w_0$ satisfying  the backward linear wave equation \eqref{eq:linear_wave_backward}
and integrate by parts. By Green's formula, we get the identity
\begin{equation}
\label{eq:integral_identity}
\begin{aligned}
&\int_0^T \int_M (V_{3}^{(1)}(t, x)-V_{3}^{(2)}(t, x))w_0w_1w_2w_3 dV_gdt
\\
&=
\int_{\Sigma \setminus \Gamma} (\partial_\nu \mathcal{U}^{(1,123)}-\partial_\nu \mathcal{U}^{(2,123)})w_0 dS_gdt.
\end{aligned}
\end{equation}

\subsection{Unique recovery of $V_3$}
\label{subsec:recovery_V3}

Our goal in this subsection is to determine the cubic nonlinearity $V_3$. To simplify the notation, in what follows we shall write $V_3(t,x):= V_{3}^{(1)}(t, x)-V_{3}^{(2)}(t, x)$. 

For any $p\in \M=(0,T)\times M$, we define the set of light-like vectors at $p$ as
\[
L_p \M = \{\zeta \in T_p \M \setminus \{0\}: g(\zeta, \zeta)=0\},
\] 
and denote by $L^\ast_p \M$ the corresponding set of light-like covectors at $p$. Moreover,  the notations $L^+_p \M$ and $L^-_p \M$ ($L^{\ast, +}_p \M$ and $L^{\ast, -}_p \M$)  stand for the sets of future and past light-like vectors (covectors), respectively. Furthermore, for any $p\in \M$ and $\theta \in L_p^{\ast, +}\M$, we define
\[
s^+(p, \theta) = \inf\{\tau>0: \gamma(\tau)\in \p \M\}
\quad \text{and} \quad 
s^-(p, \theta) = \sup\{\tau>0: \gamma(\tau)\in \p \M\}.
\]

Let $\tilde p \in \mathbb{U}$, and let $\M_1 = (0,T) \times M_1$, where $M \subset M_1^{\mathrm{int}}$.  
Then there exist covectors $\theta_0, \theta_1\in L_{\tilde p}^{\ast, +}\M$ and $\theta_2, \theta_3 \in L_{\tilde p}^{\ast, +}\M_1$, as well as constants $\kappa_j$, $j=0,1,2,3$, such that
\[
\kappa_0 \theta_0+\kappa_1 \theta_1 +\kappa_2 \theta_2 + \kappa_3 \theta_3=0,
\]
see \cite[Section 3.5]{Hintz_Uhlmann_Zhai_4th_order} for the detailed choices of $\kappa_j$ and $\theta_j$.

We set
\[
x_0 = \gamma_{\tilde p, \theta_0^\sharp} (s^+(\tilde p, \theta_0)) \in \Sigma, 
\quad 
x_1= \gamma_{\tilde p, \theta_1^\sharp} (s^-(\tilde p, \theta_0)) \in \Sigma,
\]
and denote 
\[
\xi_0 = \dot \gamma_{\tilde p, \theta_0^\sharp} (s^+(\tilde p, \theta_0)) \in L_{x_0}^- \M_1, 
\quad 
\xi_1= \dot \gamma_{\tilde p, \theta_1^\sharp} (s^-(\tilde p, \theta_0)) \in L_{x_0}^+ \M_1.
\]
Also, for $j=2,3$, let 
\[
x_j = \gamma_{\tilde p, \theta_j} (s^-(\tilde p, \theta_j)) \in \Sigma, \quad 
\xi_j = \dot \gamma_{\tilde p, \theta_j^\sharp} (s^-(\tilde p, \theta_0)) \in L_{x_j}^+ \M_1.
\]
In above, the notation $\sharp$ stands for the musical isomorphism mapping covectors to vectors.

Let us  denote $\gamma^{(j)}:= \gamma_{x_j,\xi_j}$, $j=0,1,2,3$, to be broken null geodesics such that $\gamma^{(j)}(0)=x_j$, $\dot \gamma^{(j)} (0)=\xi_j$, and $\gamma^{(j)}(s_j) = \tilde p$ for some $s_j>0$. By the null-convexity assumption of $\Sigma$, $\gamma^{(j)}$ is transversal to $\Sigma$.
%

From the constructions in Section \ref{sec:Gaussian_beam}, we obtain Gaussian beam solutions $w_j$, $j=1,2,3$, to the forward problem \eqref{eq:linear_wave_forward}, which are of the following form before the first reflection:
\[
w_j= e^{i \rho \kappa_j \varphi^{(j)}}a_{\kappa_j,\rho}^{(j)} + r_{\rho,j},
\]
Similarly, there exists a solution $w_0$ of the backward wave equation \eqref{eq:linear_wave_backward} given by 
\[
w_0= e^{i \rho \kappa_0 \varphi^{(0)}}a_{\kappa_0,\rho}^{(j)} + r_{\rho,0}.
\]
Here each remainder term $r_{\rho,j}$, $j=0,1,2,3$, satisfies the estimate \eqref{eq:est_remainder}. 

We now substitute the   solutions above into the integral identity \eqref{eq:integral_identity} and multiply both sides by $\rho^{\frac{n+1}{2}}$. By utilizing the estimate \eqref{eq:est_remainder},  the left-hand side of \eqref{eq:integral_identity} becomes
\[
\mathcal{I} := \rho^{\frac{n+1}{2}} \int_{0}^T \int_M V_3(t,x) e^{i\rho S} a_{\kappa_0, \rho}^{(0)} a_{\kappa_1, \rho}^{(1)} a_{\kappa_2, \rho}^{(2)} a_{\kappa_3, \rho}^{(3)}dV_gdt + \mathcal{O}(\rho^{-2}),
\]
where $S:=\kappa_0 \varphi^{(0)}+\kappa_1 \varphi^{(1)}+\kappa_2 \varphi^{(2)}+\kappa_3 \varphi^{(3)}$.

We next state a key lemma describing  some  properties of the function $S$, which will be essential for the subsequent analysis. This result was originally established in \cite[Lemma 5]{Feizmohammadi_Oksanen_semilinear_Euclidean}.

\begin{lemma}
\label{l.w10291} 
The function $S$ is well-defined in a small neighborhood of the point $\tilde p \in \mathbb{U}$ and  satisfies the following properties:
\begin{itemize}
\item[(i)] $S(\tilde p)=0$;
\item[(ii)] $\nabla^g S(\tilde p)=0$;
\item[(iii)] $\Im S(q) \geq  c d(q, p)^2$ for any point $q$ in a neighborhood of $\tilde p$. Here $c>0$ is a constant.
\end{itemize}
\end{lemma}

Thanks to   Lemma \ref{l.w10291},  an application of the stationary phase \cite[Theorem 7.7.5]{hormander} gives us 
\begin{equation}
\label{eq:LHS_int}
\mathcal{I}=cV_{3}(\tilde p) a_{\kappa_0, \rho}^{(0)}(\tilde p) a_{\kappa_1, \rho}^{(1)}(\tilde p) a_{\kappa_2, \rho}^{(2)}(\tilde p) a_{\kappa_3, \rho}^{(3)}(\tilde p) + \mathcal{O}(\rho^{-2})
\end{equation}
for some constant $c\ne 0$. 

We next investigate the behavior of the right-hand side of the integral identity \eqref{eq:integral_identity} as $\rho \to \infty$. Note that differentiating $u_\varepsilon=u(t,x;\varepsilon) \in E^s((0,T)\times M)$ with respect to $\varepsilon$ does not affect its regularity. Thus,  it holds that
\[
\p_\nu \mathcal{U}^{(j,123)}\big|_{\Sigma} \in   \bigcap_{k=0}^s C^k\left((0,T) ; H^{s-k-\frac{3}{2}}(\p M)\right), \quad j=1,2.
\]
In particular, $\p_\nu \mathcal{U}^{(j,123)}|_{\Sigma} \in C((0,T);H^{s-\frac{3}{2}}(\p M))$.
Since $s\ge 5$, the Sobolev embedding  $H^{s-\frac{3}{2}}(\p M) \hookrightarrow L^2(\p M)$ implies that  $\p_\nu \mathcal{U}^{(j,123)}(t,\cdot)|_{\Sigma}\in L^2(\p M)$ for each $t\in (0,T)$. Furthermore, the map $t\mapsto \|\p_\nu \mathcal{U}^{(j,123)}(t,\cdot)\|_{L^2(\p M)}$ is continuous and thus bounded. Therefore, we conclude that $\p_\nu \mathcal{U}^{(j,123)}|_{\Sigma}\in L^2(\Sigma)$.
By the Cauchy-Schwarz inequality and the estimate \eqref{eq:est_quasimode_inaccessible}, we obtain the estimate
\begin{equation}
\label{eq:RHS_int}
\left|\int_{\Sigma \setminus \Gamma} (\partial_\nu \mathcal{U}^{(1,123)}-\partial_\nu \mathcal{U}^{(2,123)})w_0 dS_gdt\right| = \mathcal{O}(\rho^{-\frac{N-k+1}{2}-\frac{3}{4}}).
\end{equation}

Taking the limit $\rho \to \infty$ in both \eqref{eq:LHS_int} and \eqref{eq:RHS_int}, we  arrive at the equation
\[
V_{3}(\tilde p) a_{\kappa_0, \rho}^{(0)}(\tilde p) a_{\kappa_1, \rho}^{(1)}(\tilde p) a_{\kappa_2, \rho}^{(2)}(\tilde p) a_{\kappa_3, \rho}^{(3)}(\tilde p)=0.
\]
Let us recall that the null geodesics $\gamma^{(j)}$, $j=0,1,2,3$ intersect only at $\tilde p$. Thus, the cut points do not exist, which implies that the product $a_{\kappa_0, \rho}^{(0)} a_{\kappa_1, \rho}^{(1)}  a_{\kappa_2, \rho}^{(2)} a_{\kappa_3, \rho}^{(3)}$ is supported near $\tilde p$. 
Therefore, we have $V_3(\tilde p) = 0$. Since $\tilde p$ is an arbitrary point in  $\mathbb{U}$, we conclude that $V_3=0$ in $\mathbb{U}$. This completes the unique recovery of $V_3$.

\subsection{Recovery of $V_m$, $m \geq 4$}
\label{subsec:recovery_higher_order}

In this subsection we aim to recover the higher-order coefficients $V_m$, $m \geq 4$. 
The proof includes performing the higher-order linearization of the partial Dirichlet-to-Neumann map $\Lambda_V^\Gamma$ and utilizing an induction argument.

Assume that the coefficients $V_3, V_4, \ldots, V_{m-1}$ have been uniquely recovered   in the set $\mathbb{U} \subset (0,T)\times M$ for some $m \geq 4$. We next show that   $\Lambda_V^\Gamma$ determines $V_m$ uniquely in $\mathbb{U}$.

We first perform the higher order linearization. Let $\varepsilon=(\varepsilon_{1},\dots,\varepsilon_{m})\in \mathbb{C}^m$, and let $u_\varepsilon=u_\varepsilon(t,x; \varepsilon) \in  E^s((0,T)\times M)$,  $s \geq 5$, be the unique  solution of the initial boundary value problem
\begin{equation}
\label{eq:higher_order}
\begin{cases}
\Box u_\varepsilon 
+\sum_{k=3}^{\infty} V_{k}(t, x) \frac{u_\varepsilon^{k}}{k!}=0  & \text { in } (0,T) \times M,
\\
u_\varepsilon  = \sum_{k=1}^m \varepsilon_k f_k & \text { on }  \Sigma,
\\
u_\varepsilon(0,\cdot)=0, \quad \p_t u_\varepsilon(0,\cdot)=0 & \text{ in }  M.
\end{cases}
\end{equation}
Let $\mathcal{U}^{(1\cdots m)} := \partial_{ \varepsilon_1} \cdots \partial_{\varepsilon_m} u_\varepsilon|_{\varepsilon=0}$. 
By carrying out similar computations as in Subsection \ref{subsec:linearization}, we see that $\mathcal{U}^{(1\cdots m)}$ satisfies the  linear equation
\begin{equation}
\label{eq:Nth_linearization}
\begin{cases}
\square_g \mathcal{U}^{(1\cdots m)} + R_m(w_1, \ldots, w_m; \mathcal{U}^{(I)}_{|I|<m}) + V_m \prod_{i=1}^m w_i = 0 & \text{in } (0,T) \times M, 
\\
\mathcal{U}^{(1\cdots m)} = 0 & \text{on } \Sigma, 
\\
\mathcal{U}^{(1\cdots m)}(0,\cdot) = \partial_t \mathcal{U}^{(1\cdots m)}(0,\cdot) = 0 & \text{in } M,
\end{cases}
\end{equation}
where $w_i = \p_{\varepsilon_i} u_\varepsilon|_{\varepsilon=0}$ is the solution to the linear wave equation \eqref{eq:linear_wave_forward} with boundary data $f_k$, and $R_m$ is a polynomial   involving the coefficients $V_3, \ldots, V_{m-1}$ and the lower-order derivatives $\mathcal{U}^{(I)}$ of $u_{\varepsilon}$ with multi-index $I$ such that $|I| < m$.

Note that the term $R_m$ depends only on coefficients $V_k$ with $k < m$ and the lower-order derivatives $\mathcal{U}^{(I)}$ with $|I| < N$, which are already known from the induction hypothesis. This follows from the structure of the nonlinearity given in the expansion \eqref{eq:expansion_V}, as well as the fact that when applying the product rule to the function $\partial_{ \varepsilon_1} \cdots \partial_{ \varepsilon_m} V(x,u_\varepsilon)$, any term involving $V_m$ must contain exactly one factor of $V_m$ and $m$ factors of first derivatives $\partial_{\varepsilon_k} u_\varepsilon$, while terms involving $V_k$ with $k < m$ contain more than $m$ total derivatives.  

Indeed, we first observe that $\partial_{\varepsilon_1} \cdots \partial_{\varepsilon_m} (\sum_{k=m+1}^{\infty} V_k(t,x)\frac{u_\varepsilon^k}{k!})$ is a sum of terms each containing positive powers of $u_\varepsilon$, which   vanish when $\varepsilon=0$. 
Moreover, the only term in $\partial_{\varepsilon_1} \cdots \partial_{\varepsilon_m} (V_m(t,x)\frac{u_\varepsilon^m}{m!})$ that does not contain a positive power of $u_\varepsilon$ is $V_m \prod_{i=1}^m w_i$. 
Finally, the expression $\partial_{\varepsilon_1} \cdots \partial_{\varepsilon_m} (\sum_{k=3}^{m-1} V_k(t,x)\frac{u_\varepsilon^k}{k!})|_{\varepsilon=0}$ contains only derivatives of $u_\varepsilon$ of the form $\partial_{\varepsilon_{l_1}\ldots\varepsilon_{l_k}}^{k}u_\varepsilon|_{\varepsilon=0}$ with $k=1,\dots, m-1$ and  $\varepsilon_{l_1}, \dots, \varepsilon_{l_k} \in \{\varepsilon_1, \dots, \varepsilon_m\}$. By the induction hypothesis and the unique solvability of the Dirichlet problem for the wave equation, we have that $\partial_{\varepsilon_{l_1}\ldots\varepsilon_{l_k}}^{k}u_1|_{\varepsilon=0} = \partial_{\varepsilon_{l_1}\ldots\varepsilon_{l_k}}^{k}u_2|_{\varepsilon=0}$ for $k=1,\ldots,m-1$. 
Hence, the term 
\[
R_m(t,x) := \partial_{\varepsilon_1} \cdots \partial_{\varepsilon_m} \left(\sum_{k=3}^{m-1} V_k(t,x)\frac{u_\varepsilon^k}{k!}\right)\bigg|_{\varepsilon=0}
\] 
is independent of $u_j$, $j=1,2$. Here $u_j$ satisfies the problem \eqref{eq:higher_order} with the coefficient $V^{(j)}$.

Consider two functions $V^{(1)}$ and $V^{(2)}$ such that $\Lambda_{V^{(1)}}^\Gamma = \Lambda_{V^{(2)}}^\Gamma$. By the induction hypothesis, we have $V_k^{(1)} = V_k^{(2)}$ for $k = 3, \dots, m-1$, and consequently $R_m^{(1)} = R_m^{(2)}$. Let $w_0$ be a solution to the backward wave equation \eqref{eq:linear_wave_backward}. Multiplying first equation in \eqref{eq:Nth_linearization} by $w_0$ and integrating by parts, we argue similarly as in the case $m=3$ to deduce that
\[
\int_0^T \int_M \left(V_m^{(1)}(t,x) - V_m^{(2)}(t,x)\right) w_0 \prod_{k=1}^m w_k  dV_g dt = \int_{\Sigma \setminus \Gamma} \left(\partial_\nu \mathcal{U}^{(1,1\cdots m)} - \partial_\nu \mathcal{U}^{(2,1\cdots m)}\right) w_0  dS_g dt.
\]

Let $\tilde p \in \mathbb{U}$, and let $\gamma^{(j)}$, $j=0,1,2,3$, be null geodesics as in Subsection \ref{subsec:recovery_V3}. We choose the Gaussian beam solutions $w_0, w_1, \ldots, w_m$, which are given by the following expression before the first reflection:
\[
w_j= e^{i \rho \kappa_j \varphi^{(j)}}a_{\kappa_j,\rho}^{(j)} + r_{\rho,j}, \quad j=0,1,2,
\]
and
\[
w_j= e^{i \rho \frac{\kappa_3}{N-2} \varphi^{(j)}}a_{ \frac{\kappa_3}{N-2},\rho}^{(j)} + r_{\rho,j}, \quad j=3, \dots, N,
\]
where each remainder term $r_{\rho,j}$ satisfies the estimate \eqref{eq:est_remainder}. 
By the same arguments as in Subsection \ref{subsec:recovery_V3}, we get that
\[
V_m(t,x) a_{\kappa_0, \rho}^{(0)}(\tilde p) a_{\kappa_1, \rho}^{(1)}(\tilde p) a_{\kappa_2, \rho}^{(2)}(\tilde p) \left(a_{\kappa_3, \rho}^{(3)}(\tilde p)\right)^{m-2}=0.
\]
Let us recall   that the function $a_{\kappa_0, \rho}^{(0)}(\tilde p) a_{\kappa_1, \rho}^{(1)}(\tilde p) a_{\kappa_2, \rho}^{(2)}(\tilde p) (a_{\kappa_3, \rho}^{(3)}(\tilde p))^{m-2}$ is supported near $\tilde p$. Thus, it follows that $V_m(\tilde p)=0$. Since $\tilde p$ is arbitrary, it holds that $V_m(t,x)=0$ in $\mathbb{U}$.

Therefore, by induction, we have $V_m^{(1)} = V_m^{(2)}$ for all $m \geq 3$. Furthermore, since the power series for $V(t,x,z)$ converges in the $C^\alpha$-topology, we conclude that $V^{(1)}(t,x,z) = V^{(2)}(t,x,z)$ for all $(t,x)\in (0,T)\times M$ and $z \in \mathbb{C}$. This completes the proof of Theorem \ref{thm:main_result}.

\section*{Acknowledgments}

The authors would like to express their gratitude to Ali Feizmohammadi, Katya Krupchyk, Gunther Uhlmann, and Yang Zhang for their helpful discussions. B.L. was partially supported by the  Simons Foundation Travel Support for Mathematicians (MPS-TSM-00013766). W.W. was partially supported by the  Simons Foundation Travel Support for Mathematicians (No. 0007730).

\bibliographystyle{abbrv}
\bibliography{bib_nonlinear_hyperbolic}

\end{document}